\documentclass[11pt]{amsart}
\usepackage[colorlinks=true,citecolor=magenta,linkcolor=magenta, backref=page]{hyperref}
\usepackage{amssymb,verbatim,amscd,amsmath,graphicx}
\usepackage{enumerate}
\usepackage{graphicx}
\usepackage{caption}
\usepackage{subcaption}
\usepackage{pdfsync}
\usepackage{color,colortbl}
\usepackage{fullpage}
\usepackage{bbm}
\usepackage{comment}
\usepackage{multirow}
\numberwithin{equation}{section}
\newtheorem*{introthm*}{Theorem}
\newtheorem{theorem}{Theorem}[section]
\newtheorem{lemma}[theorem]{Lemma}
\newtheorem{proposition}[theorem]{Proposition}
\newtheorem{corollary}[theorem]{Corollary}

\newtheorem{claim}[theorem]{Claim}

\theoremstyle{definition}
\newtheorem{definition}[theorem]{Definition}
\newtheorem{conjecture}[theorem]{Conjecture}
\newtheorem{def-prop}[theorem]{Definition-Proposition}

\newtheorem{remark}[theorem]{Remark}
\newtheorem{example}[theorem]{Example}

\newtheorem*{acknowledgement}{Acknowledgements}

\DeclareMathOperator{\reg}{reg}

\DeclareMathOperator{\Ass}{Ass}
\DeclareMathOperator{\Min}{Min}

\renewcommand{\AA}{{\mathbb A}}
\newcommand{\CC}{{\mathbb C}}

\newcommand{\PP}{{\mathbb P}}

\newcommand{\NN}{{\mathbb N}}

\newcommand{\XX}{{\mathbb X}}

\newcommand{\kk}{{\mathbbm k}}

\def\L{{\mathcal L}}

\def\mm{{\frak m}}

\def\pp{{\frak p}}

\def\a{{\bf a}}

\def\p{{\bf p}}

\def\x{{\bf x}}

\def\z{{\bf z}}

\def\ahat{\widehat{\alpha}}

\def\1{{\bf 1}}
\def\0{{\bf 0}}

\begin{document}
	
\title{Chudnovsky's Conjecture and the stable Harbourne-Huneke containment for general points}

\author{Sankhaneel Bisui}
\address{University of Manitoba \\ Department of Mathematics \\
	Machray Hall 420, 186 Dysart Rd.\\ Winnipeg, MB R3T 2M8, CA}
\email{Sankhaneel.Bisui@umanitoba.ca} 

\author{Th\'ai Th\`anh Nguy$\tilde{\text{\^e}}$n}
\address{Tulane University, Department of Mathematics,
	6823 St. Charles Ave., New Orleans, LA 70118, USA \\
	and University of Education, Hue University, 34 Le Loi St., Hue, Viet Nam}
\email{tnguyen11@tulane.edu}

\keywords{Chudnovsky's conjecture, Cremona Transformation, Waldschmidt Constant, Ideals of Points, Symbolic Powers, Containment Problem, Stable Harbourne-Huneke Conjecture}
\subjclass[2010]{14N20, 13F20, 14C20}

\begin{abstract}
In our previous work with Grifo and H\`a, we showed the stable Harbourne-Huneke containment and Chudnovsky's conjecture for the defining ideal of sufficiently many \emph{general} points in $\PP^N$. In this paper, we establish the conjectures for \textit{all} remaining cases, and hence, give the affirmative answer to Harbourne-Huneke containment and Chudnovsky's conjecture for \textit{any number of general points} in $\PP^N$ for all $N$. Our new technique is to develop the Cremona reduction process that provides effective lower bounds for the Waldschmidt constant of the defining ideals of \emph{generic} points in projective spaces. 
\end{abstract}
\maketitle
\section{introduction} \label{sec.intro}
In the work for providing counterexamples to Hilbert's $14^{th}$-problem Nagata asked the following question: \emph{Take a set of reduced points $\XX=\{P_1,  \dots, P_s\} \subset \mathbb{P}^2_{\mathbb{C}}.$ What is the minimal degree $\alpha_m(\XX)$ of a hypersurface that passes through the given points with multiplicity at least $m$?} Nagata conjectured that for at least $10$ \emph{general} points, $\alpha_m(\XX) \geqslant m\sqrt{s}, \text{ for each } m \geqslant 1$, and proved it for $k^2$ many general points (the open condition depends on $m$). The conjecture is still wide open and a vast number of papers in the last few decades are related to this conjecture. Later on, Iarrobino \cite{iarrobino_1997} conjectured that $\alpha_m (\XX) \geqslant m\sqrt[N]{s},$ for sufficiently large number of \emph{general} points in $\PP^N$. The only known evidence for this conjecture due to Evain \cite{EVAIN2005516}, for $s=k^N$ many general points, when $N \geqslant 3, k \geqslant 3$. These conjectures are equivalent to saying that all the inequalities (for all $m$) hold for (sufficiently many) \textit{very general} points.

On the other hand, interests for the study of $\alpha_m (\XX)$ came from other various contexts. We refer interested readers to \cite{CHHVT2020} for more information. A more classical motivation of this study is in the context of complex analysis, see \cite{Chudnovsky1981}, \cite{Moreau1980}. In particular, there have been various studies to get effective lower bounds for $\alpha_m(\XX)$. Waldschmidt \cite{Waldschmidt} and Skoda \cite{Skoda} proved the inequality $\dfrac{\alpha_m(\XX)}{m} \geqslant \dfrac{\alpha(\XX)}{N}$ for points in $\PP^N_\CC$ using complex analytic techniques where $\alpha(\XX)$ denotes the least degree of a hypersurface that passes through the points at least one time. Chudnovsky\cite{Chudnovsky1981} improved the bound for points $\PP^2_{\CC}$, by proving that
$\dfrac{\alpha_m(\XX)}{m} \geqslant \dfrac{\alpha(\XX)+1}{2}.$ 
In the same paper, he conjectured the following inequality for a general set of points in $\PP^N_\CC$, 
\begin{conjecture}\cite{Chudnovsky1981}
	If $\XX=\{P_1,  \dots, P_s\} \subset \mathbb{P}^N_{\mathbb{C}} $, then
	$$\dfrac{\alpha_m(\XX)}{m} \geqslant \dfrac{\alpha(\XX)+N-1}{N}, \text{ for all } m \geqslant 1.$$
\end{conjecture} 

All these geometric problems can be re-stated in an algebraic way using the well-celebrated Zariski-Nagata Theorem (\cite{Zariski,Nagata,EisenbudHochster}). More precisely, finding lower bounds for $\alpha_m(\XX)$ is equivalent to searching for lower bounds for $\alpha \big(I^{(m)} \big)$, where $I$ is the defining ideal of $\XX$, $I^{(m)}$ denotes the $m$-th symbolic power of $I$, and $\alpha(J)$ denotes the initial degree of a homogeneous ideal $J$. Thus, Chudnovsky's conjecture takes the following equivalent format:
\begin{conjecture}[Chudnovsky's Conjecture]
	Let $\XX=\{P_1,  \dots, P_s\} \subset \mathbb{P}^N_{\mathbb{C}} $ and $I$ be the defining ideal of $\XX$. Then   
	$$\dfrac{\alpha\big(I^{(m)} \big)}{m} \geqslant \dfrac{\alpha(I)+N-1}{N}, \text{ for all } m \geqslant 1.$$
\end{conjecture} 
The containment problem of symbolic and ordinary powers of ideals is very well-studied (see e.g., \cite{HaHu, Seceleanu, GrifoHuneke, GrifoStable,   DrabkinSeceleanu, bghn2021chudnovskys, bghn2022demailly, nguyen2021least, nguyen2021leastlike,SteinerConfigurations}.) One of the important applications to study these containment is the fact that the containment would provide lower bounds on the initial degree of the symbolic powers. Consider the following celebrated Theorem by Ein-Lazarsfled-Smith and Hoschter-Huneke:

\begin{theorem}{\cite{ELS,comparison}} \label{theorem: containment Ihm in Im} 
	For	a radical ideal $I$ of big height $h$ in a regular ring $S$, one has  $I^{(hm)} \subseteq I^m$ for all $m\in \NN$. 
\end{theorem}

If $I$ is a defining ideal of points in $\PP^N_\CC$, then Theorem \ref{theorem: containment Ihm in Im} implies $\frac{\alpha \left(I ^{(m)} \right)}{m} \geqslant \frac{\alpha(I)}{N}, \text{ for all } m \geqslant 1,$ which is the bound proved by Waldschmidt and Skoda. To strengthen the containment, Harbourne-Huneke conjectured that for a homogeneous radical ideal $I \subset \kk[\PP^N_\kk]$ of big height $N$, one would expect that  $I^{(mN)} \subseteq \mm^{m(N-1)}I^m$ for all $m \geqslant 1$, where $\mm = \langle x_0, x_1, \dots, x_N \rangle $. Chudnovsky's conjecture follows from stable version of the containment, which has been studied in \cite{bghn2021chudnovskys}. 

\begin{conjecture}[Stable Harbourne-Huneke containment]
	Let $I \subseteq \kk[\PP^N_\kk]$ be a homogeneous radical ideal of big height $h$. Then there exists a constant $r(I) \geqslant 1$, depending on $I$, such that for all $r \geqslant r(I)$, we have
	$$(1) \quad I^{(hr)} \subseteq \mm^{r(h-1)}I^r \qquad \textrm{ and } \qquad (2) \quad I^{(hr-h+1)} \subseteq \mm^{(r-1)(h-1)}I^r.$$
\end{conjecture}

Previously, the stable Harbourne-Huneke containment  $I^{(hr)} \subseteq \mm^{r(h-1)}I^r $, and hence, Chudnovsky's conjecture had been shown in the following cases: any set of points in $\PP^2_\kk$ \cite{HaHu}, a \emph{general} set of points in $\PP^3_\kk$ \cite{dumnicki2012symbolic, Dumnicki2015}, a set of at most $N+1$ points in \emph{generic position} in $\PP^N_\kk$ \cite{Dumnicki2015}, a set of points forming a \emph{star configuration} \cite{BoH, GHM2013}. In addition, Chudnovsky's conjecture is known for a set of points in $\PP^N_\kk$ lying on a quadric \cite{FMX2018}, and a \emph{very general} set of points in $\PP^N_\kk$ \cite{DTG2017, FMX2018}. By saying that a property $\mathcal{P}$ holds for a \emph{very general} set of points in $\PP^N_\kk$, we mean that there exist infinitely many open dense subsets $U_m$, $m \in \NN$, of the Hilbert scheme of $s$ points in $\PP^N_\kk$ such that the property $\mathcal{P}$ holds for all $\XX \in \bigcap_{m=1}^\infty U_m$. If we remove this infinite intersection of open dense subsets and show that there exists one open dense subset $U$ of the Hilbert scheme of $s$ points in $\PP^N_\kk$ such that the property $\mathcal{P}$ holds for all $\XX \in U$, then the property $\mathcal{P}$ holds for a \emph{general} sets of points. Informally, while very general properties correspond to (intersection of) countable open conditions, general properties correspond to one open condition. \par
\vspace{0.5em}

The stable Harbourne-Huneke containment and Chudnovsky's conjecture was shown to hold for at least $3^N$ many general points when $N\geqslant 4$, and the number of points in the results can be reduced to at least $2^N, \text{when } N \geqslant 9$ in \cite{bghn2021chudnovskys}. The key idea in the proof is that a stronger containment, namely, $I^{(hr-h)} \subseteq \mm^{r(h-1)}I^r, r \gg 0$, would imply Harbourne-Huneke stable containment. In \cite{bghn2021chudnovskys}, this stronger containment has been proved for a \emph{sufficiently large} number of \emph{generic} points, utilizing the important inequality $\ahat(I) > \frac{\reg(I)+N-1}{N}$, where $\ahat(I)$ is the \emph{Waldschmidt constant}, defined by $\ahat(I) := \lim_{m \rightarrow \infty} \frac{\alpha(I^{(m)})}{m}.$ This required an appropriate lower bound for $\ahat(I)$, but unfortunately, the method in \cite{bghn2021chudnovskys} could only provide such bounds for sufficiently large (exponential) numbers of generic points, but not for smaller numbers of points.\par
\vspace{0.5em}

In this manuscript, we use \emph{Cremoma transformation} to provide a reduction process to get desired lower bounds for $\ahat(I)$ of the defining ideals of generic points. Our strategy, inspired from the works \cite{DumnickiAlgorithm, dumnicki2012symbolic, Dumnicki2015}, is to reduce the study of lower bounds for Waldschmidt constants of defining ideals of generic points to that of a fewer number of generic points. More precisely, we use Cremona transformation as our primary tool to show the following.

\begin{introthm*}[Theorem \ref{theorem: reduction on Waldschmidt consant} and Proposition \ref{proposition: 2^N reduces to 2}]
If  $\ahat(s) = \ahat(I (1^{\times s}))$, and $I \big(1^{\times {b \cdot 2^N}}, \overline{m}\big)$ denotes the defining  ideal of  $b \cdot2^N+s$ generic points, where $b \cdot 2^N$ have multiplicity 1 and the remaining $s$ points have multiplicities $m_1, \dots, m_s$ respectively, then 
$$(1) \quad \ahat\big(b \cdot(2^N)^k\big) \geqslant 2^k\ahat (b) \qquad \textrm{ and } \qquad (2) \qquad \ahat\big( I \big(1^{\times b\cdot 2^N}, \overline{m}\big) \big) \geqslant \ahat\big( I \big(2^{\times b}, \overline{m} \big)\big). $$
\end{introthm*}

As a result of this reduction process combined with a similar approach using \emph{specialization} as in \cite{bghn2021chudnovskys}, see also \cite{bghn2022demailly}, yields the results on the stable Harbourne-Huneke containment and Chudnovsky's conjecture for a small number of general points. Combining this and previous results on sufficiently many general points, we are able to complete the picture for all numbers of general points. One key point of the proof is the appropriate lower bound on Waldschmidt constant of generic points.

\begin{introthm*}[Theorem \ref{theorem: lowerboundsforallnumber}]
Let $I$ be the defining ideal of any number of $s$ generic points in $\PP^N$ where $s\geqslant N+4$. Then $$\ahat(I) > \frac{\reg(I)+N-1}{N}.$$ 
\end{introthm*}
Note that the Waldschmidt constant for defining ideals of up to $N+3$ generic points are computed in \cite{brianlinear} and Harbourne-Huneke Containment as well as Chudnovsky's Conjecture would follow easily, see also \cite{NagelTrokInterpolation}. Hence, we are interested in ideals defining at least $N+4$ generic points when $N\geqslant 4$. The main result of this paper is the affirmed answer to the stable Harbourne-Huneke Containment and Chudnovsky's Conjecture for any number of general points in any dimensional projective spaces.
\begin{introthm*}[Theorem \ref{theorem: Containment for general points} and Theorem \ref{theorem: Chudnovsky for small points}]
Then ideal defining a set of any number of $s$ general points in $\PP^N$ satisfies the stable Harbourne-Huneke Containment, and hence, satisfies Chudnovsky's Conjecture.
Furthermore, there is a constant $r(s,N)$ depends only on $s$ and $N$ such that the containment $I^{(Nr)} \subseteq \mm^{(N-1)r}I^r$ hold when $I$ is the defining ideal of $s$ general points and $r \geqslant r(s,N)$.
\end{introthm*}

The paper is outlined as follows. Section \ref{sec.prel} introduces necessary terminology and notations and recalls some valuable results. In Section \ref{sec.reduction}, we establish Theorems regarding Cremona transformation and obtain lower bounds on the Waldschmidt constant of ideals defining small numbers of generic fat points. In Section \ref{section: lowerboundWaldschmidt}, we establish the important lower bound for the Waldschmidt constant of generic points. In Section \ref{section: Chudnovsky}, we prove the stable Harbourne-Huneke containment and Chudnovsky's conjecture for any numbers of general points. 

\begin{acknowledgement}
The first author is thankful to Adam Van Tuyl for asking him questions regarding Chudnovsky's Conjecture when $N=4$ during his talk at the Canadian Mathematical Society Winter Meeting on December '21, which led to this manuscript. Both authors are thankful to Marcin Dumnicki, Huy T\`ai  H\`a, Paolo Mantero, and Alexandra Seceleanu for valuable suggestions. The first author was partially funded by the Faculty of Science and Department of Mathematics at the University of Manitoba.
\end{acknowledgement}

\section{Preliminaries} \label{sec.prel}
 We introduce basic notations and known results that we will be using throughout the paper. We will work with the assumption that $N\geqslant 4$ as both the stable \emph{Harbourne-Huneke containment}, and \emph{Chudnovsky's conjecture} for any sets of general points are known for $N=2$ (see \cite{HaHu}) and $N=3$ (see \cite{dumnicki2012symbolic, Dumnicki2015}). We also use the umbrella assumption that $\kk$ is any algebraically closed field. $S= \kk[\PP^N_{\kk}]$ represents the homogeneous coordinate ring of the projective space $\PP^N_\kk$. Our work focus on symbolic powers, the Waldschmidt constant, and Cremona transformations, so we define them individually.
 \begin{definition}
 	Let $R$ be a commutative ring and let $I \subseteq R$ be an ideal. For $m \in \NN$, the \emph{$m$-th symbolic power} of $I$ is defined to be
 	$$I^{(m)} = \bigcap_{\pp \in \Ass(I)} \left(I^mR_\pp \cap R\right).$$
 \end{definition}
 
We remark here that there is also a notion of symbolic powers in which the set $\Min(I)$ of minimal primes is used in place of the set $\Ass(I)$ of associated primes in the definition. In the context of this paper, for defining ideals of points, or, more generally, ideals with no embedded primes, these two notions of symbolic powers agree. It is well-known that if $\XX$ is the set $ \{P_1, \dots, P_s\} \subseteq \PP^N_\kk$ of $s$ many distinct points and let $\p_i \subseteq \kk[\PP^N_\kk]$ be the defining ideal of $P_i$ and  $I = \p_1 \cap \dots \cap \p_s$ is the ideal defining $\XX$. Then the $m$-th symbolic power is given by, 
	$$I^{(m)} = \p_1^m \cap \dots \cap \p_s^m.$$
 
 \begin{definition}
 	If  $I \subseteq \kk[\PP^N_\kk]$ is homogeneous ideal and $\alpha(I)$ denotes its least generating degree, then the \emph{Waldschmidt constant} of $I$ is defined as 
 	$$\ahat(I) := \lim_{m \rightarrow \infty} \dfrac{\alpha(I^{(m)})}{m} = \inf_{m \in \NN} \dfrac{\alpha(I^{(m)})}{m}.$$ 
 	See, for example, \cite[Lemma 2.3.1]{BoH}. 
 \end{definition}
 Using the Waldschmidt constant of defining ideal of set of points in $\PP^N_\kk$, Chudnovsky's conjecture takes the following format.
 \begin{conjecture}[Chudnovsky]
 	Let $I \subseteq \kk[\PP^N_\kk]$ be the defining ideal of a set of (reduced) points in $\PP^N_\kk$. Then,
 	$$\ahat(I) \geqslant \dfrac{\alpha(I)+N-1}{N}.$$
 \end{conjecture}
 
 \begin{definition}
 	Let $\p_i $ denotes the ideal defining a point $P_i \in \XX=\{P_1, \dots P_s \} \subset \PP^N_\kk$ and $\overline{m}=(m_1, \dots m_s)$ is a sequence of positive integers. Then the fat point scheme denoted by $m_1P_1+m_2P_2+ \dots + m_sP_s$ is the scheme defined by the ideal 
 	$$I (\overline{m})= I(m_1, \dots m_s) = \p_1^{m_1} \cap \p_2^{m_2} \cap \dots \cap \p_s^{m_s}. $$
 	If $m_j \leqslant 0$, then we take $\p_j^{m_j}=\kk[\PP^N_\kk]$. 
 	We will also use the following notation: $$m^{\times s}=\underbrace{(m,m, \dots, m)}_{ s \text{ times }}.$$
 \end{definition}

Let $\XX=\{P_1, \dots, P_s \} \subset \PP^N_\kk$ be a set of points. Then $\L_N(d; m_1, \dots, m_s)$ denotes the linear system of hypersurfaces of degree $d$ passing though the $s$ points $P_1, \dots, P_s$ with multiplicity $m_1, m_2, \dots m_s$, respectively. In our context, $\L_N(d; m_1, \dots, m_s) = [I(m_1, \dots m_s)]_d$, the degree $d$-component of the defining ideal.

 \begin{definition}
 	The standard birational transformation 
 	$$\Phi: \PP^N_\kk \to \PP^N_\kk, \text{ defined by } \Phi(x_0: \dots: x_N) \mapsto (x_0^{-1}: \dots: x_N^{-1}),$$
 	is known as Cremona transformation. 
 \end{definition}
 The following Lemma is due to \cite[Theorem 3]{DumnickiAlgorithm}, see also, \cite[Lemma B.1.2]{brianlinear}, which infers how Cremona operations do not alter the linear system up to a certain degree of adjustment. The Lemmas were originally shown for points in general position, but the proof applies for generic points or general points as well. We restate the theorems in our context of defining ideals. 
 \begin{lemma}\label{lemma: cremoa from brian}
 	For $N\geqslant 2$, the Cremona transformation $(x_0: \dots x_N) \mapsto (\dfrac{1}{x_0}: \dots: \dfrac{1}{x_N} )$ of $\PP^N$ induces a linear isomorphism 
 	$$ [I(m_1, \dots m_s)]_d \longrightarrow [I(m_1+k, \dots,m_{N+1}+k,m_{N+2}, 
 	\dots,  m_s)]_{d+k} $$
 	provided that $m_i+k \geqslant 0$, for $i=1,\ldots , N+1$, where $k=(N-1)d-\sum_{j=1}^{N+1}m_j$.
 \end{lemma}
 
 The following Lemmas, due to \cite[Theorem 4]{DumnickiAlgorithm} and \cite[Proposition 10]{Dumnicki2015}, are very helpful in our reduction process. Our assumption for the set of points is still generic or general.
 
 \begin{lemma}\cite[Theorem 4]{DumnickiAlgorithm} \label{lemma: Theorem 4 from Dumnicki Algorithm}
 	Let $N \geqslant 2$, let $d, m_1, m_2, \dots, m_r \in \NN$. If $(N-1)d- \sum_{j=1}^N m_j < 0, m_j >0$ for $j=1, \dots, N$ then 
 	$$\dim [I(m_1, \dots m_s)]_d=\dim [I( m_1-1, \dots, m_N-1, m_{N+1}, \dots m_r )]_{d-1}.$$
 \end{lemma}

 \begin{lemma}{\cite[Proposition 10]{Dumnicki2015}} \label{lemma: adding multiplicities Dumnicki}
 	Let  $m_1, \dots, m_r, m_1', \dots, m'_s, t,k$ be integers. If $I(m_1, \dots m_r)_k =0$ and 	$I(m_1', \dots, m'_s, k+1 )_t =0,$
 	then $I(m_1, \dots, m_r, m_1', \dots, m'_s )_t=0 .$
 \end{lemma}
 
 We also recall some well known results about Waldschmidt constants of defining ideals of small number of points, see also \cite{NagelTrokInterpolation}.
 \begin{lemma}\label{lemma: known inequalities of Waldschmidt constant}
 	If	$\ahat(s) = \ahat(I (1^{\times s}))$  is the Waldschmidt constant the defining ideal of $s$ generic points in $\PP^N$, then the followings are true  
 	\begin{enumerate}
 		\item $\ahat(s) \geqslant \ahat(k)$ whenever $s \geqslant k$;  
 		\item  $\alpha(I(m^{\times s})) \geqslant m\ahat(s) $;
 		\item $ \ahat \big(I \big( 1 ^{\times {k^N}}\big) \big) =k$,  More precisely, $I(m^{\times {k^N}})_{km-1}=0$, when $k \geqslant 2$ \cite{EVAIN2005516, DTG2017}.
 	\end{enumerate}
 \end{lemma}
 
 \begin{proposition} \cite[Proposition B.1.1]{brianlinear} \label{proposition: known bounds on Wladschmidt constant upto N+3 points}
 	If $I(1^{\times s})$ denotes the ideal defining $s$ many generic points in $\PP^N$, then 
 	\begin{enumerate}
 		\item $\ahat\big(I \big( 1^{\times (N+1)}\big) \big) \geqslant \dfrac{N+1}{N}$;
 		\item $\ahat\big(I \big( 1^{\times (N+2)}\big) \big) \geqslant \dfrac{N+2}{N}$;
 		\item $\ahat\big(I \big( 1^{\times (N+3)}\big) \big) \geqslant \dfrac{N+2}{N}$ if $N$ is even;
 		\item $\ahat\big(I \big( 1^{\times (N+3)}\big) \big) \geqslant 1+\dfrac{2}{N}+\dfrac{2}{N^3+2N^2-N}$ if $N$ is odd.
 	\end{enumerate}
 \end{proposition}
 We have mentioned generic and general points many times before. Now we recall some facts about specialization, generic and general points in $\PP^N_\kk(\z)$. The set of all collections of $s$ not necessarily distinct points in $\PP^N_\kk$ is parameterized by the \emph{Chow variety} $G(1,s,N+1)$ of $0$-cycles of degree $s$ in $\PP^N_\kk$ (cf. \cite{GKZ1994}). Thus, a property $\mathcal{P}$ is said to hold for a \emph{general} set of $s$ points in $\PP^N_\kk$ if there exists an open dense subset $U \subseteq G(1,s,N+1)$ such that $\mathcal{P}$ holds for any $\XX \in U$.
 
 Let $(z_{ij})_{1 \leqslant i \leqslant s, 0 \leqslant j \leqslant N}$ be $s(N+1)$ new indeterminates. We shall use $\z$ and $\a$ to denote the collections $(z_{ij})_{1 \leqslant i \leqslant s, 0 \leqslant j \leqslant N}$ and $(a_{ij})_{1 \leqslant i \leqslant s, 0 \leqslant j \leqslant N}$, respectively. Let
 $$P_i(\z) = [z_{i0}: \dots : z_{iN}] \in \PP^N_{\kk(\z)} \quad \text{ and } \quad \XX(\z) = \{P_1(\z), \dots, P_s(\z)\}.$$
 The set $\XX(\z)$ is often referred to as the set of $s$ \emph{generic} points in $\PP^N_{\kk(\z)}$. For any $\a \in \AA^{s(N+1)}_\kk$, let $P_i(\a)$ and $\XX(\a)$ be obtained from $P_i(\z)$ and $\XX(\z)$, respectively, by setting $z_{ij} = a_{ij}$ for all $i,j$. There exists an open dense subset $W_0 \subseteq \AA^{s(N+1)}_\kk$ such that $\XX(\a)$ is a set of distinct points in $\PP^N_\kk$ for all $\a \in W_0$ (and all subsets of $s$ points in $\PP^N_\kk$ arise in this way). The following result allows us to focus on open dense subsets of $\AA^{s(N+1)}_\kk$ when discussing general sets of points in $\PP^N_\kk$.
 
 \begin{lemma}[\protect{\cite[Lemma 2.3]{FMX2018}}] \label{lem.Hilbert}
 	Let $W \subseteq \AA^{s(N+1)}_\kk$ be an open dense subset such that a property $\mathcal{P}$ holds for $\XX(\a)$ whenever $\a \in W$. Then, the property $\mathcal{P}$ holds for a general set of $s$ points in $\PP^N_\kk$.
 \end{lemma}
 
 \begin{definition}[Krull] \cite[Definition 2.8]{bghn2021chudnovskys}
 	Let $\x$ represent the coordinates $x_0, \dots, x_N$ of $\PP^N_\kk$. Let $\a \in \AA^{s(N+1)}$. The \emph{specialization} at $\a$ is a map $\pi_\a$ from the set of ideals in $\kk(\z)[\x]$ to the set of ideals in $\kk[\x]$, defined by
 	$$\pi_\a(I) := \{f(\a,\x) ~\big|~ f(\z,\x) \in I \cap \kk[\z,\x]\}.$$
 \end{definition}
 
 \begin{remark} \cite[ Remark 2.9]{bghn2021chudnovskys}
 	Let $\pp_i(\z)$ and $\pp_i(\a)$ be the defining ideals of $P_i(\z) \in \PP^N_{\kk(\z)}$ and $P_i(\a) \in \PP^N_\kk$, respectively. It follows from \cite[Satz 1]{Krull1948} that there exists an open dense subset $W \subseteq W_0 \subseteq \AA^{s(N+1)}$ such that, for all $\a \in W$ and any $1 \leqslant i \leqslant s$, we have
 	$$\pi_\a(\pp_i(\z)) = \pp_i(\a).$$
 	We shall always assume that $\a \in W$ whenever we discuss specialization in this paper.
 \end{remark}
 
 \begin{remark} \cite[Remark 2.10]{bghn2021chudnovskys} \label{rmk.KrullSpecialization}
 	Observe that, by the definition and by \cite[Satz 2 and 3]{Krull1948} (see also \cite[Propositions 3.2 and 3.6]{NhiTrung1999}), for fixed $m, r, t \in \NN$, there exists an open dense subset $U_{m,r,t} \subseteq W$ such that for all $\a \in U_{m,r,t}$, we have
 	$$\pi_\a\left(I(\z)^{(m)}\right) = I(\a)^{(m)} \text{ and } \pi_\a\left({\mm}_{\z}^t I(\z)^r\right) = \mm^t I(\a)^r.$$
 	Here, we use $\mm$ and $\mm_\z$ to denote the maximal homogeneous ideals of $\kk[\x]$ and $\kk(\z)[\x]$, respectively. Note that $\mm_\z$ is the extension of $\mm$ in $\kk(\z)[\x]$. We shall make use of this fact often.
 \end{remark}
 
\section{Reduction Process and Lower Bound on Waldschmidt Constant for Small Numbers of Points} \label{sec.reduction}
We start this section by a consequence of Lemma \ref{lemma: cremoa from brian} and Lemma \ref{lemma: Theorem 4 from Dumnicki Algorithm}. The following result will be our essential tool to get appropriate lower bounds on the Waldschmidt constant. 
\begin{lemma} \label{lemma: reduction of multiplicity of points}
	Let $I(m_1, \dots, m_s)$ denote the ideal of $s$ generic points or general points with multiplicities $m_1, \dots, m_s$ respectively. If $$I(m_1, \dots, m_s)_d \neq 0, \text{ then } I (m_1+k, \dots,m_{N+1}+k,m_{N+2}, 
	\dots,  m_s ) _{d+k} \neq 0,$$ where $k = (N-1)d-\sum_{j=1}^{N+1}m_j $. 
\end{lemma}
The spirit of the proof is the same as \cite[Proposition 8]{dumnicki2012symbolic}.
\begin{proof}
	We prove the Theorem case by case. Without any loss of generality we can assume that $m_1 \geqslant m_2 \geqslant m_3 \geqslant \dots \geqslant m_{N+1}$. 
	\begin{itemize}
		\item If $m_{N+1}+k \geqslant 0$, then the conclusion follows from Lemma \ref{lemma: cremoa from brian}.   
		\item Suppose $l=m_{N+1}+k < 0$ and $m_{N}+k\geqslant 0$. Now $(N-1)d-\sum_{i=1}^{N}m_i = m_{N+1}+k=l$. Then by using the hypothesis and repeated application of Lemma \ref{lemma: Theorem 4 from Dumnicki Algorithm}, we will get 
		$$I\big(m_1+l, \dots, m_N+l, m_{N+1}, \overline{m} \big)_{d+l} \neq 0, $$
		where $\overline{m}=(m_{N+2},\ldots ,m_s)$. Note that $m_N + l \geqslant 0$, hence $m_i+l\geqslant 0$ for $i=1,\ldots , N$. In fact, suppose that $m_N + l <0$, then, $m_N +l=m_{N}+m_{N+1}+k=(N-1)d-\sum_{i=1}^{N-1}m_i <0$, which is a contradiction because the last inequality would imply $m_1>d$. Now $d+l < m_{N+1}$ will lead to a contradiction so $d+l \geqslant m_{N+1}$. 
		Again we compute, $k'= (N-1)(d+l)- \sum_{j=1}^N(m_j+l) -m_{N+1}=k-l$. Thus $m_{j}+l+k'=m_j+k \geqslant 0$ for $j=1, \dots N$. Also $m_{N+1}+k' =0$. Thus by Lemma \ref{lemma: cremoa from brian} we get 
		$I\big(m_1+k, \dots, m_N+k, 0, \overline{m} \big)_{d+k} \neq 0, $
		and since $m_{N+1} +k < 0$, we can write 
		$$I\big(m_1+k, \dots, m_N+k, m_{N+1}+k, \overline{m} \big)_{d+k} \neq 0. $$
		
		\item Suppose $l=m_{N+1}+k <0$, $l'=m_{N}+k<0$, and $k+m_j \geqslant 0$, for $j=1, \dots N-1$. Note that $l'\geqslant l$, hence, as in the previous case, by Lemma \ref{lemma: Theorem 4 from Dumnicki Algorithm}, we will get, 
		$$I\big(m_1+l, \dots, m_N+l, m_{N+1}, \overline{m} \big)_{d+l} \neq 0. $$
		Note that $m_i + l \geqslant 0$ for $i=1,\ldots ,N$ by the same argument as above case. We also have that $d+l \geqslant m_{N+1}$. Now, $(N-1)(d+l) -\sum_{j=1}^{N-1}(m_j + l)-m_{N+1}=m_N+k=l'.$ Thus, again by Lemma \ref{lemma: Theorem 4 from Dumnicki Algorithm}, we get that 
		$$I\big(m_1+l+l', \dots,m_{N-1}+l+l', m_N+l, m_{N+1}+l', \overline{m} \big)_{d+l+l'} \neq 0 $$
		Indeed, it is enough to check that $m_{N+1}+l' \geqslant 0$ and $m_i+l+l' \geqslant 0$ for $i=1, \ldots , N-1$. Firstly, $m_{N+1}+l' = m_N+l \geqslant 0$. Secondly, we see that $$m_{N-1}+l+l'\geqslant m_{N-1}+m_{N+1}-d+l' = (N-2)d-\sum_{j=1}^{N-2}m_j \geqslant 0$$
		since otherwise, $m_1<d$, which is a contradiction.\\
		\noindent Now, again, $k' = (N-1)(d+l+l')- \sum_{j=1}^{N-1}(m_j+l+l')-(m_N+l)-(m_{N+1}+l')=k-l-l'$.
		Now, $m_j+l+l'+k' =m_j+k \geqslant 0$, for $j=1, \dots, N-1$. Also, $m_N+l+k'=0$, and $m_{N+1}+k' =0$. Thus by Lemma \ref{lemma: cremoa from brian}, 
		$\big(m_1+k, \dots, m_{N-1}+k, 0,0, \overline{m} \big)_{d+k} \neq 0, $
		and since, $m_N+k <0$, and $m_{N+1}+k <0$, we can write, 
		$$I\big(m_1+k, \dots, m_N+k, m_{N+1}+k, \overline{m} \big)_{d+k} \neq 0. $$ 
		
		\item Suppose $l_i=m_i +k < 0$, for $i= j, \dots N+1$ and $m_{i}+k \geqslant 0$ for $i=1, \dots, j-1.$ Note that $l_1 \geqslant l_2 \geqslant \ldots \geqslant l_{N+1}$ and $m_i+l_j=m_j+l_i$ for all $i,j$. Proceed as above cases, we will get $$I\big(m_1+l_{N+1}, \dots, m_N+l_{N+1}, m_{N+1}, \overline{m} \big)_{d+l_{N+1}} \neq 0. $$
		as long as $m_N+l_{N+1} \geqslant 0$. This is true by exactly the same argument as in first case, that is,
		$m_N+l_{N+1}=(N-1)d-\sum_{i=1}^{N-1}m_i \geqslant 0$. Now, as before, note that $l_{N+1}\geqslant m_{N+1}-d$, and $(N-1)(d+l_{N+1}) -\sum_{j=1}^{N-1}(m_j + l_{N+1})-m_{N+1}=m_N+k=l_N<0,$ apply Lemma \ref{lemma: Theorem 4 from Dumnicki Algorithm} we will get 
		$$I\big(m_1+l_{N+1}+l_{N}, \dots,m_{N-1}+l_{N+1}+l_N, m_N+l_{N+1}, m_{N+1}+l_N, \overline{m} \big)_{d+l+l'} \neq 0 $$
		as long as all the multiplicities appeared above are nonnegative. This is true exactly by the argument as in case 2, that is, $m_{N+1}+l_N= m_{N+1}+l_N\geqslant 0$, and that
		$$m_{N-1}+l_{N+1}+l_N\geqslant m_{N-1}+m_{N+1}-d+l_N = (N-2)d-\sum_{j=1}^{N-2}m_j \geqslant 0.$$
		Now by repeated application of Lemma \ref{lemma: Theorem 4 from Dumnicki Algorithm} ($N+2-j$ times) we will get, 
		$$I\big(m_1+\sum_{i=j}^{N+1} l_i, \dots m_{j-1}+\sum_{i=j}^{N+1} l_i , m_j+\sum_{i \neq j} l_i,  \dots, m_{N+1}+\sum_{i \neq N+1} l_i , \overline{m} \big)_{d+\sum_{i=j}^{N+1} l_i} \neq 0. $$
		Now, $k' = (N-1)(d+\sum_{i=j}^{N+1} l_i)- \sum_{i'=1}^{j-1} \big(m_{i'}+\sum_{i=j}^{N+1} l_i \big) - \sum_{i'=j}^{N+1}  \big(  m_{i'}+\sum_{i \neq i'} l_i  \big) = k - \sum_{i=j}^{N+1}l_i.$
		At the $t+1$ step of the process, the multiplicity
		$$m_{N-t}+\sum_{j=N+t-1}^{N+1} l_j \geqslant (N-t-1)d-\sum_{j=1}^{N-t-1}m_j \geqslant 0$$
		and other multiplicities are nonnegative from the $t$-step.
		Now, note that, $m_{i'}+\sum_{i=j}^{N+1} l_i+k' =m_{i'}+k$, for $i'= 1 \dots j-1$. Also, note that 
		$m_{i'}+\sum_{i \neq i'} l_i+k'  =0$, for $i' =j, \dots, N+1$. Thus by applying Cremona transformation e.g., Lemma \ref{lemma: cremoa from brian} we get, 
		$$I\big(m_1+k, \dots, m_{j-1}+k, \underbrace{0,\dots, 0}_{N+2-j \text{ many}}, \overline{m} \big)_{d+l} \neq 0. $$
		Since $ m_i +k < 0$, for $i= j, \dots N+1$, then we can write 
		$$I\big(m_1+k, \dots, m_{j-1}+k, m_j+k, \dots, m_{N+1}+k, \overline{m} \big)_{d+l} \neq 0. $$
	\end{itemize}
\end{proof}

The following Theorem, inspired by \cite[Proposition 10]{dumnicki2012symbolic} and \cite[Proposition 12]{Dumnicki2015}, along with the Proposition \ref{proposition: 2^N reduces to 2} shown below are the main tools in our reduction process.

\begin{theorem}\label{theorem: reduction on Waldschmidt consant}
	If  $\ahat(s) = \ahat(I (1^{\times s}))$  is the Waldschmidt constant the defining ideal of $s$ generic points in $\PP^N$, then
	$$\ahat\big(b \cdot(2^N)^k\big) \geqslant 2^k\ahat(b), \text{ where } b \text{ and } k \text{ are positive integers}. \label{inequality: main}$$
\end{theorem}
\begin{proof}
	To prove the inequality, first we prove that $\ahat(2^Nb)\geqslant 2\ahat(b)$ in Claim \ref{claim: base inequality}, then we use induction in Claim \ref{claim: inductive stage} to get the result.
	
	\begin{claim} \label{claim: base inequality}
		If $\ahat$ is described as in the statement, then  $\ahat(2^Nb)\geqslant 2\ahat(b)$.
	\end{claim}
	\begin{proof}
		From Lemma \ref{lemma: known inequalities of Waldschmidt constant}, we have 
		$\ahat(2^N)=2$, hence, $I(m^{\times 2^N})_{2m-1}=0$. If $I((2m)^{\times b})_{t} =0$, then applying Lemma \ref{lemma: adding multiplicities Dumnicki} one time we get that
		$I(m^{\times2^N}, (2m)^{\times(b-1)})_t =0.$
		If we keep applying Lemma \ref{lemma: adding multiplicities Dumnicki} one more times we get, 
		$I(m^{\times2^N}, m^{\times2^N}, (2m)^{\times(b-2)})_t =0$
		Thus, by applying Lemma \ref{lemma: adding multiplicities Dumnicki} total $b$ times we get $I(m ^{\times b \cdot 2^N } )_t =0. $ Therefore $$\alpha\big(I (m ^{\times b \cdot 2^N } ) \big) \geqslant \alpha \big( I((2m)^{\times b}) \big) \geqslant 2m \ahat(b). $$
		Now by dividing each sides by $m$, and taking limit as $n \to \infty$, we get
		$\ahat(2^Nb)\geqslant 2\ahat(b). $
	\end{proof}
	\begin{claim}\label{claim: inductive stage}
		If $\ahat(2^Nb)\geqslant 2\ahat(b)$, then $ \ahat\big(b\cdot(2^N)^k\big) \geqslant 2^k\ahat(b)$. 
	\end{claim}
	\begin{proof}
		We prove this by using induction algorithm. The statement is true for $k=1$. Assume that it is true for $k$. 
		Then $\ahat\big(b\cdot(2^N)^{k+1}\big) =\ahat\big (2^N (b\cdot(2^N)^k)\big) \geqslant 2\ahat\big(b\cdot(2^N)^k\big) =2\cdot 2^{k}\ahat\big(b\big) =2^{k+1}\ahat(b), $
		hence the claim.
	\end{proof}
	This ends the proof of the Theorem.
\end{proof}

\begin{corollary}\label{corollary: 2^k comes outside}
	If $\ahat(s) =\ahat(I(1^{\times s}))$, denotes the Waldschmidt constant of the defining ideal of $s$ generic points in $\PP^N$, and $s\geqslant b \cdot (2^N)^k$, then $$\ahat(s) \geqslant 2^k \ahat(b). $$
\end{corollary}
\begin{proof}
	This is straightforward from Theorem \ref{theorem: reduction on Waldschmidt consant}, and Lemma \ref{lemma: known inequalities of Waldschmidt constant}. 
\end{proof}
\begin{proposition}\label{proposition: 2^N reduces to 2}
	If $\overline{m} =(m_1, \dots, m_s)$ is a sequence of multiplicities,  then 
	$$ \ahat\big( I \big(1^{\times 2^N}, \overline{m}\big) \big) \geqslant \ahat\big( I \big(2, \overline{m} \big)\big),$$
	where $I \big(1^{\times 2^N}, \overline{m}\big)$ denotes the defining  ideal of  $2^N+s$ generic points, where $2^N$ have multiplicity 1 and the remaining $s$ points have multiplicities $m_1, \dots, m_s$, respectively. As a consequence, 
	$$ \ahat\big( I \big(1^{\times b\cdot 2^N}, \overline{m}\big) \big) \geqslant \ahat\big( I \big(2^{\times b}, \overline{m} \big)\big).$$
	
\end{proposition}
\begin{proof}
	The proof follows by the same argument as \cite[Theorem 9]{dumnicki2012symbolic}.
	By Lemma \ref{lemma: known inequalities of Waldschmidt constant} we know that $ I \big(m ^{\times 2^N} \big)_{2m-1} =0$.
	Let  $J= I\big(2, \overline{m} \big)$ and $I= I \big(1^{\times 2^N}, \overline{m} \big)$. Suppose, $J^{(m)}_t =0$. Then by Lemma \ref{lemma: adding multiplicities Dumnicki}, we get 
	$I^{(m)}_t =0.$ Thus $\alpha(I^{(m)}) \geqslant \alpha(J^{(m)})$, which implies that
	$$ \ahat\big( I \big(1^{\times 2^N}, \overline{m}\big) \big) \geqslant \ahat\big( I \big(2, \overline{m} \big)\big).$$
	The latter inequality follows directly by successively applying the above inequality $b$ times
	$$ \ahat\big( I \big(1^{\times b\cdot 2^N}, \overline{m}\big) \big) = \ahat\big( I \big(1^{\times 2^N}, 1^{\times (b-1)\cdot 2^N}, \overline{m}\big) \big) \geqslant \ahat\big( I \big(2, 1^{\times (b-1)\cdot 2^N}, \overline{m}\big) \big) \geqslant \cdots \geqslant \ahat\big( I \big(2^{\times b}, \overline{m} \big)\big).$$
\end{proof}

\begin{remark}
	
	Theorem \ref{theorem: reduction on Waldschmidt consant} and Proposition \ref{proposition: 2^N reduces to 2} suggest that to find lower bounds of $\ahat(s)$, one can hope to reduce the number of points and work on getting lower bounds for a fewer number of points. Sometimes, this reduction gives very useful bounds, as in the following example. 
\end{remark}

\begin{example}
	Consider 128 generic points in $\PP^4$. Then by Proposition \ref{proposition: 2^N reduces to 2}
	$$\ahat(128)\geqslant \ahat(8 \cdot 16) \geqslant 2\ahat(8) \geqslant \dfrac{16}{5}.$$
	The last inequality follows from Lemma \ref{lemma: lower bounds on Waldschmidt constant in P^4}, part $(1)$. We can see that bounds on the Waldschmidt constant of the defining ideal of 8 generic points can be useful to get bounds on 128 generic points. Also note that the new bound  in fact is better than  bound $\ahat(128) \geqslant 3$, which is obtained from the inequality $81 \leqslant 128 \leqslant 256$ and Lemma \ref{lemma: known inequalities of Waldschmidt constant}. 
\end{example}

After reducing the number of points, we need to obtain appropriate lower bounds on the Waldschmidt constant of some small number of points as well. The following results provide us with what we need to proceed in the next section.

\begin{lemma}\label{lemma: lower bounds on Waldschmidt constant in P^4} Let $I(m_1, \dots, m_s)$ be the defining ideal of $s$ generic points with multiplicities, $m_1, \dots, m_s$ in $\PP^4$.  The following inequalities hold
	\begin{enumerate}
		\item  $\ahat(\big(I(1^{\times 8}) \big)) \geqslant \dfrac{8}{5} $;
		\item $\ahat\big(I \big( 2^{\times 4}, 1^{\times 7} \big) \big) \geqslant \dfrac{23}{10}$;
		\item $\ahat\big(I \big(1^{\times 36} \big) \big) \geqslant \dfrac{51}{25}$.
	\end{enumerate}
\end{lemma}
\begin{proof}
	We use Lemma \ref{lemma: reduction of multiplicity of points} to prove the bounds. The idea is inspired from  \cite[Proposition 11]{dumnicki2012symbolic} and \cite[Proposition 11]{Dumnicki2015}.
	\begin{enumerate}
		
		
		\item Suppose that $I \big( (5m)^{\times 8}_{8m-1} \big) \neq 0$. We show reduction by repeated application of Lemma \ref{lemma: reduction of multiplicity of points} in the following table which leads to a contradiction.
		\begin{table}[h!]
			\centering
			\begin{tabular}{||c |c c c c c c  c c| c ||} 
				\hline
				d & $m_1$ & $m_2$ & $m_3$ & $m_4$ & $m_5$ & $m_6$ & $m_7$ & $m_8$& k \\ [0.5ex] 
				\hline\hline
				$8m-1$ & $\underline{5m}$ & $\underline{5m}$ & $\underline{5m}$ &$\underline{5m}$  & $\underline{5m}$ & $5m$ & $5m$ & $5m$ &$-m-3$ \\ 
				$7m-4 $ & $ \underline{4m-3}$ & $ \underline{4m-3}$& $4m-3$& $4m-3$& $4m-3$& $\underline{5m}$ & $ \underline{5m}$& $ \underline{5m}$& $-2m-6$\\
				$5m-10 $ &  $2m-9$&$2m-9$&  $\underline{4m-3}$& $ \underline{4m-3}$& $\underline{4m-3}$& $ \underline{3m-6}$& $ \underline{3m-6}$& $ 3m-6$& $-3m-9$ \\
				$2m-19$ &$2m-9 $ & $2m-9$ & $m-12$ & $m-12$& $m-12$& & & $3m-6$ &\\
				\hline
			\end{tabular}
		\end{table}\\
		\	From the last row, $I \big( (2m-9)^{\times 1} \big)_{2m-19} \neq0 $, a contradiction. 
		Hence, $\ahat(\big(I(1^{\times 8}) \big)) \geqslant \dfrac{8}{5} .$

		\item Suppose that $I \big( (20m)^{\times 4}, (10m)^{\times 7} \big)_{23m-1} \neq 0$
		and set $k=3(23m-1)-(4 \cdot 20m +10m)=-21m-3$. Hence by Lemma \ref{lemma: reduction of multiplicity of points}, $I\big(0^{\times 5}, (10m)^{6}  \big)_{2m-4} \neq 0$. which is a contradiction.
		Hence, $\ahat\big(I \big( 2^{\times 4}, 1^{\times 7} \big) \big) \geqslant \dfrac{23}{10}.$
		
		\item We show that $I\big((25m)^{\times 36 })\big)_{51m-1} =0$ for all $m$. If $I\big( (25m)^{\times 4},50m,(40m)^{\times 2} \big)_{51m-1} \not=0$, then by Lemma \ref{lemma: reduction of multiplicity of points} with $k=-27m-3$, we have $I\big( (25m)^{\times 2},23m-3,(13m-3m)^{\times 2} \big)_{24m-4} \not=0$, which is a contradiction. Hence, $I\big( (25m)^{\times 4},50m,(40m)^{\times 2} \big)_{51m-1} =0$. Combine this with the fact that $I\big((25m)^{\times 8 })\big)_{40m-1} =0$ for all $m$ (since $\ahat(\big(I(1^{\times 8}) \big)) \geqslant \dfrac{8}{5} $ by Lemma \ref{lemma: lower bounds on Waldschmidt constant in P^4}), we get $I\big( (25m)^{\times 12},50m,(40m) \big)_{51m-1} =0$ by Lemma \ref{lemma: adding multiplicities Dumnicki}. Apply Lemma \ref{lemma: adding multiplicities Dumnicki} again, we have $I\big( (25m)^{\times 20},50m \big)_{51m-1} =0$. Lastly, combine this with $I\big( (25m)^{\times 16} \big)_{50m-1} =0$ and apply Lemma \ref{lemma: adding multiplicities Dumnicki} yet again, we get $I\big((25m)^{\times 36 })\big)_{51m-1} =0$ for all $m$.
	\end{enumerate}
\end{proof}

\begin{lemma}\label{lemma: lower bounds on Waldschmidt constant in P^5}
	Let $I(m_1, \dots, m_s)$ be the defining ideal of $s$ generic points with multiplicities, $m_1, \dots, m_s$ in $\PP^5$. The following inequalities hold
	$$\ahat\big( I\big( 2^{\times 3}, 1^{\times 31} \big) \big) \geqslant \dfrac{21}{10}.$$
\end{lemma}
\begin{proof}
	We will again use Lemma \ref{lemma: reduction of multiplicity of points} to get desired lower bounds. Suppose that\\
	$I\big( (20m)^{\times 3}, (10m)^{31} \big)_{21m-1} \neq 0$ and set $k=4(21m-1)-(3\cdot 20m + 3 \cdot 10m)=-6m-4$. By Lemma \ref{lemma: reduction of multiplicity of points}, $I\big((14m-4)^{\times 3}, (4m-4)^{\times3}, (10m)^{28}\big)_{15m-5} \neq 0.$ Now set $k=4(15m-5)-(3 \cdot (14m-4)+3 \cdot 10m)=-12m-8$, then again by Lemma \ref{lemma: reduction of multiplicity of points}, we get
	$I \big((2m-12)^{\times 3}, (4m-4)^{\times 3}, (10m)^{\times 25} \big)_{3m-13}\neq 0,$ which is a contradiction. Hence, 
	$\ahat\big( I\big( 2^{\times 3}, 1^{\times 31} \big) \big) \geqslant \dfrac{21}{10}.$
\end{proof}

\begin{lemma}\label{lemma: lower bounds on Waldschmidt constant in P^N}
	Let $I$  denotes the ideal of $N+4$ generic points in $\PP^N$, where $N$ is an even number and $N\geqslant 6.$ Then 
	$$\ahat \big(I \big) \geqslant  \dfrac{(N+2)(2N-1)+2}{N(2N-1)}.$$
\end{lemma}
\begin{proof}
	Let $q_1=\dfrac{mN(2N-1)}{2}$, and $p_1=\dfrac{m(N+2)(2N-1)}{2}+m-1$. Suppose that $$I\big( q_1 ^{\times(N+4)}  \big) _{p_1} \neq 0.$$
	Set $k_1=(N-1)p_1-(N+1)q_1=-mN-(N-1)$. 
	By Lemma \ref{lemma: reduction of multiplicity of points}, 
	$ I\big( ( q_2^{\times(N+1)},q_1^{\times 3}   \big) _{p_2} \neq 0,$
	where, $$q_2=q_1-mN-(N-1)= \dfrac{N(2N-3)m}{2}-(N-1), \text{ and } p_2=p_1-mN-(N-1)=\dfrac{N(2N+1)m}{2}-N.$$	
	Now, $k_2=(N-1)p_2-\big( (N-1)q_2+2q_1 \big)=-mN-(N-1).$ Applying Lemma \ref{lemma: reduction of multiplicity of points}, 
	$$I \big( q_3^{ \times(N-1)}, q_2^{\times 2}, q_4^{\times 2}, q_1  \big)_{p_3} \neq 0 $$
	where, $ q_3 = q_2-mN-(N-1), q_4=q_1-mN-(N-1)$, and $p_3=p_2-mN-(N-1)=\dfrac{mN(2N-1)}{2}-(2N-1)$, which is a contradiction as $q_1>p_3$. Thus we get, 
	$$ \ahat \big(I \big) \geqslant \dfrac{(N+2)(2N-1)+2}{N(2N-1)}. $$
\end{proof}

\section{Lower Bound for Waldschmidt Constant} \label{section: lowerboundWaldschmidt}
In this section, we show the key inequality $\ahat(I) > \frac{\reg(I)+N-1}{N}$ where $I$ is a defining ideal of any number of generic points in $\PP^N$, which is the crucial point in the proof of Stable Harbourne-Huneke Containment. We will combine the results in \cite{bghn2021chudnovskys} for sufficiently many points, the reduction process in section \ref{sec.reduction}, and the bounds on Waldschmidt constant of defining ideals of a small number of points given in Lemma \ref{lemma: lower bounds on Waldschmidt constant in P^4}, Lemma \ref{lemma: lower bounds on Waldschmidt constant in P^5}, and Lemma \ref{lemma: lower bounds on Waldschmidt constant in P^N} to obtain the needed bound on the Waldschmidt constant. \par
\vspace{0.5em}
First, notice that if the number of generic points $s$ satisfies ${N+\ell-1 \choose N} < s \leqslant {N+\ell \choose N}$, then it is well-known that by \cite[Lemma 5.8]{MN2001} and \cite[Corollary 1.6]{GM1984}, we have $\reg(I) = \ell+1$. Therefore, the inequality is equivalent to $\ahat(I) > \frac{d+1}{N},$ where ${d \choose N} < s \leqslant {d+1 \choose N}$ for all $d\geqslant N-1$. Note also that since we are interested in the case when $s\geqslant N+4$, we can assume that $d \geqslant N$. Finally, the inequality was proved for sufficiently many generic points in $\PP^N$ in \cite{bghn2021chudnovskys}, in particular, for at least $3^N$ general points when $N \geqslant 4$, and for at least $2^N$ general points when $N \geqslant 9$. The following lemma shows the inequality for all unknown cases in $\PP^4$.

\begin{lemma}\label{lemma: lowers bounds on Waldschmidt in P4 for 8<s< 81 }
	Let $I$  be the defining ideal of $s$ generic points in $\PP^4$, and  $ 8 \leqslant s \leqslant 81$. Then $\ahat\big( I ( 1^{ \times s } ) \big) > \dfrac{d+1}{4}  $, whenever ${d \choose 4} < s \leqslant {d+1 \choose 4}$. 
\end{lemma}
\begin{proof}
	We divide into different cases. Since $ s \leqslant 81$ we start with $  s \leqslant  {9 \choose 4}$ and proceed. 
	\begin{enumerate}
		\item When $d=8$,  
		then by using Lemma \ref{lemma: known inequalities of Waldschmidt constant}, Proposition \ref{proposition: 2^N reduces to 2}, and Lemma \ref{lemma: lower bounds on Waldschmidt constant in P^4}, we get $$\ahat\big(I(1^{\times s })\big) \geqslant \ahat\big(I(1^{\times 71})\big) \geqslant  \ahat\big(I(1^{\times (16
			\cdot 4 +7) })\big) \geqslant  \ahat\big(I(2^{\times 4}, 1^{\times 7 })\big) \geqslant 23/10 > 9/4.$$
		
		\item  When $d=7$, then by Lemma \ref{lemma: known inequalities of Waldschmidt constant}, and Lemma \ref{lemma: lower bounds on Waldschmidt constant in P^4}, 
		$\ahat\big(I(1^{\times s })\big) \geqslant \ahat\big(I(1^{\times 36})\big) \geqslant 51/25 > 8/4.$
		
		\item When $d=6$, 
		then by using Lemma \ref{lemma: known inequalities of Waldschmidt constant} we get 
		$\ahat\big(I(1^{\times s })\big) \geqslant \ahat\big(I(1^{\times 16})\big) \geqslant  2 > 7/4.$
		
		\item When $ 8 \leqslant  s \leqslant  {6 \choose 4}$,
		then Lemma \ref{lemma: known inequalities of Waldschmidt constant} and \ref{lemma: lower bounds on Waldschmidt constant in P^4} we get 
		$\ahat\big(I(1^{\times s })\big) \geqslant \ahat\big(I(1^{\times 8})\big) \geqslant  8/5 > 6/4.$
	\end{enumerate}
\end{proof}
From now on, we only work with $N\geqslant 5$. The next lemma reduces the number of points to at least $2^N$ for all $N\geqslant 5$.
\begin{lemma} \label{lemma: lowers bounds on Waldschmidt in P5678 for 2^N<s< 3^N}
	Let $I$ be the defining ideal of $s$ generic points in $\PP^N$, where $2^N \leqslant s \leqslant 3^N$, and  $N=5,6,7$, and $8$. Then $\ahat\big( I ( 1^{ \times s } ) \big) > \dfrac{d+1}{N}  $ whenever ${d \choose N} < s \leqslant {d+1 \choose N}$. 
\end{lemma}
\begin{proof}
	We prove individually for $N=5, 6, 7$, and $8$ by dividing into sub-cases and proving them. 
	\begin{enumerate}
		\item Consider $N=5$ and $2^5 \leqslant s \leqslant 3^5$. Since $s\leqslant 3^5$, then  $s \leqslant {10 \choose 5}$. Now we study case by case:
		\begin{enumerate}
			\item If ${9 \choose 5} < s \leqslant {10 \choose 5}$, then by Lemma \ref{lemma: known inequalities of Waldschmidt constant}, Proposition \ref{proposition: 2^N reduces to 2}, and Lemma \ref{lemma: lower bounds on Waldschmidt constant in P^5}
			$$\ahat(I(1^{\times s})) \geqslant \ahat \big(I (1^{\times 126}) \big) = \ahat \big(I \big(1^{\times(3 \times 32+31) }) \big)\geqslant \ahat\big(I(2 ^{\times 3}, 1^{\times 31} \big) \big)\geqslant \dfrac{21}{10}>\dfrac{10}{5} . $$
			
			\item If $d=7\text{ or }8$ with $s\geqslant 32$, by Lemma  \ref{lemma: known inequalities of Waldschmidt constant}, $\ahat(I(1^{\times s}))   \geqslant  \ahat \big(I \big(1^{\times 32} )\big)= 2  > \dfrac{d+1}{5}.$
			
			
		\end{enumerate}
		
		\item Consider $N=6$ and $2^6 \leqslant s \leqslant 3^6$. Since $s\leqslant 3^6$, then  $s \leqslant {12 \choose 6}$.
		\begin{enumerate}
			\item If ${11 \choose 6} < s \leqslant {12 \choose 6}$, then using Lemma \ref{lemma: known inequalities of Waldschmidt constant}, Theorem \ref{theorem: reduction on Waldschmidt consant} and Proposition \ref{proposition: known bounds on Wladschmidt constant upto N+3 points} we get:
			$$\ahat(I(1^{\times s})) \geqslant \ahat \big(I (1^{\times  462}) \big) \geqslant  \ahat \big(I \big(1^{\times(7\cdot 64)}\big) \geqslant 2 \ahat \big(I (1 ^{\times 7}) \big) \geqslant  2 \cdot \dfrac{7}{6} > \dfrac{12}{6}.$$
			
			\item If $d= 8, 9, 10$, with $s\geqslant 64$, by Lemma \ref{lemma: known inequalities of Waldschmidt constant}, $\ahat(I(1^{\times s}))   \geqslant  \ahat \big(I \big(1^{\times 64} )\big)= 2  > \dfrac{d+1}{6}.$
			
			
			
		\end{enumerate}	
		
		\item Consider $N=7$ and $2^7 \leqslant s \leqslant 3^7$. Since $s\leqslant 3^7$, then  $s \leqslant {14 \choose 7}$.
		
		\begin{enumerate}
			\item If ${13 \choose 7} < s \leqslant {14 \choose 7}$, then  by Lemma \ref{lemma: known inequalities of Waldschmidt constant}, Theorem \ref{theorem: reduction on Waldschmidt consant}, and Proposition \ref{proposition: known bounds on Wladschmidt constant upto N+3 points}
			$$\ahat(I(1^{\times s})) \geqslant \ahat \big(I (1^{\times  1716}) \big) \geqslant  \ahat \big(I \big(1^{\times(8\cdot 128)} \big) \big)\geqslant 2 \ahat \big(I (1 ^{\times 8}) \big) \geqslant  2 \cdot \dfrac{8}{7} > \dfrac{14}{7}.$$
			
			\item  If $d=10,11,12$ with $s\geqslant 128$, by \ref{lemma: known inequalities of Waldschmidt constant}, $\ahat(I(1^{\times s}))   \geqslant  \ahat \big(I \big(1^{\times 128} )\big)= 2  > \dfrac{d+1}{7}.$
			
			
			
		\end{enumerate}
		
		\item Consider $N=8$ and $2^8 \leqslant s \leqslant 3^8$. Since $s\leqslant 3^8$, then  $s \leqslant {16 \choose 8}$.
		\begin{enumerate}
			\item If $d=15 \text{ or } 16$, then by Lemma \ref{lemma: known inequalities of Waldschmidt constant} and Theorem \ref{theorem: reduction on Waldschmidt consant}
			$$\ahat(I(1^{\times s})) \geqslant \ahat \big(I (1^{\times  6435}) \big) \geqslant  \ahat \big(I \big(1^{\times(9\cdot 256)}\big) \geqslant 2 \ahat \big(I (1 ^{\times 9}) \big) \geqslant  2\cdot \dfrac{9}{8}  > \dfrac{d+1}{8}.$$ 
			\item  If $d=11, 12, 13, 14$, with $s\geqslant 256$, by \ref{lemma: known inequalities of Waldschmidt constant}, 
			$\ahat(I(1^{\times s})) \geqslant \ahat \big(I (1^{\times  256}) \big) =2 > \dfrac{d+1}{8}.$
		\end{enumerate}
	\end{enumerate}
\end{proof}

Since we work with the number of points bounded between binomial numbers, ${d \choose N} < s \leqslant {d+1 \choose N}$, the following numerical lemma allows us to make the assumption that $d \leqslant 2N-2$. 

\begin{lemma}\label{lemma: combinatorialineq}
Let $N\geqslant 3$, then $\displaystyle {2N-1 \choose N} \geqslant 2^N$.    
\end{lemma}
\begin{proof}
    The proof is straightforward by induction.
\end{proof}
In the next two lemmas, we show the inequality $\ahat(I) > \frac{\reg(I)+N-1}{N}$ for $s$ generic points where ${d \choose N} < s \leqslant {d+1 \choose N}$ in the cases when $d=N+1$, and $d=N+2$. 
\begin{lemma} \label{lemma: lowers bounds on Waldschmidt in PN for N+4<s< N+2chooseN}
	Let $I$ be the defining ideal of $s$ generic points in $\PP^N$, where $N+4 \leqslant s \leqslant {N+2 \choose N}$, and  $N \geqslant 5$. Then $\ahat\big( I ( 1^{ \times s } ) \big) > \dfrac{N+2}{N}$.  
\end{lemma}
\begin{proof}
	Let $N+4 \leqslant s \leqslant { N+2 \choose N }$, if $N$ is even then by Lemma \ref{lemma: known inequalities of Waldschmidt constant} and Lemma \ref{lemma: lower bounds on Waldschmidt constant in P^N} $$\ahat \big(I \big(1^{\times s} \big) \big) \geqslant \ahat \big(I \big(1^{\times (N+4)} \big) \big) \geqslant \dfrac{(N+2)(2N-1)+2}{N(2N-1)}> \dfrac{N+2}{N}.$$
	Otherwise, if $N$ is odd, then by Lemma \ref{lemma: known inequalities of Waldschmidt constant} and Proposition \ref{proposition: known bounds on Wladschmidt constant upto N+3 points} we have $$\ahat \big(I \big(1^{\times s} \big) \big) \geqslant \ahat \big(I \big(1^{\times (N+3)} \big) > \dfrac{N+2}{N}.$$
	This finishes the proof.
\end{proof}

\begin{lemma} \label{lemma: lowers bounds on Waldschmidt in PN for N+2chooseN<s< N+3chooseN}
	Let $I$ be the defining ideal of $s$ generic points in $\PP^N$, where ${N+2 \choose N} \leqslant s \leqslant {N+3 \choose N}$, and  $N \geqslant 5$. Then $\ahat\big( I ( 1^{ \times s } ) \big) > \dfrac{N+3}{N}$.  
\end{lemma}
\begin{proof}
It is enough to show the inequality for $s={N+2 \choose N}+1$ points. We make use of the fact that $\ahat({1^{\times N+2}})\geqslant \frac{N+2}{N}$, hence, $$\displaystyle I((Nam)^{\times N+2})_{(N+2)am-1} = 0, \text{  for all  } a,m,$$
to show that there exist $x,y$, and $a$ (for each $N$) such that $x+(N+2)y={N+2 \choose N}+1$, and 
$$\displaystyle I\left((Nam)^{\times x},((N+2)am)^{\times y}\right)_{[(N+3)a+1]m-1} = 0, \forall m.$$

In fact, once we prove the above claim, by applying the Lemma \ref{lemma: adding multiplicities Dumnicki} $y$ times, we get
$$ \displaystyle I((Nam)^{\times s})_{[(N+3)a+1]m-1} = 0,\forall m,$$
and therefore, $\ahat({1^{\times s}}) \geqslant \dfrac{(N+3)a+1}{Na} > \dfrac{N+3}{N}$. We prove the claim in two cases as follows. \par
\vspace{0.5em}

\textbf{Case 1: When $N$ is odd}
\begin{claim}
Take $x=N+3$, $y=\dfrac{N-1}{2}$, and any $a\geqslant \dfrac{N^2-1}{2(N-2)}$.
\end{claim}

\begin{proof}
We need to show that 
$$\displaystyle I\left((Nam)^{\times N+3},((N+2)am)^{\times \frac{N-1}{2}}\right)_{[(N+3)a+1]m-1} = 0, \forall m.$$

Applying Lemma \ref{lemma: reduction of multiplicity of points} with
\begin{align*}
    k&= (N-1)[(N+3)a+1]m-(N-1)- \left(\dfrac{N-1}{2}(N+2)am+\dfrac{N+3}{2}Nam\right) \\
    &= \left((N-1)-2a\right) m-(N-1),
\end{align*}
(From here, we need $a\geqslant \frac{N+1}{2}$.)

for all $m$, we reduce to
$$I\left(A_1^{\times \frac{N+3}{2}},B_1^{\times \frac{N+3}{2}},C_1^{\times \frac{N-1}{2}}\right)_{D_1},$$
where $A_1 = Nam, B_1= [(N-2)a+N-1]m-(N-1)$, $C_1=(Na+N-1)m-(N-1) $, and $D_1=[(N+1)a+N]m-N$.\par
\vspace{0.5em}

Applying Lemma \ref{lemma: reduction of multiplicity of points} again with
\begin{align*}
    k&= (N-1)D_1- \dfrac{N-1}{2}C_1 -\dfrac{N+3}{2}A_1 \\
    &= \left(\dfrac{N^2-1}{2}-(N+1)a\right) m- \dfrac{N^2-1}{2},
\end{align*}

for all $m$, we reduce to

$$I\left( A_2^{\times \frac{N+3}{2}},B_2^{\times \frac{N+3}{2}}, C_2^{\times \frac{N-1}{2}}\right)_{D_2},$$

where $A_2=(\dfrac{N^2-1}{2}-a)m- \dfrac{N^2-1}{2},$ $B_2=[(N-2)a+N-1]m-(N-1) $,\\ $C_2=(\dfrac{N^2-1}{2}+N-1-a)m- \dfrac{N^2-1}{2}-(N-1)$, and $D_2=(\dfrac{N^2-1}{2}+N)m-\dfrac{N^2-1}{2}-N$.\par
\vspace{0.5em}
Therefore, if we choose any $a\geqslant \dfrac{N^2-1}{2(N-2)}$, we have 
$$(N-2)a+(N-1) \geqslant \dfrac{N^2-1}{2}+N , \text{  hence,  } B_2>D_2,$$
thus, $I\left( A_2^{\times \frac{N+3}{2}},B_2^{\times \frac{N+3}{2}}, C_2^{\times \frac{N-1}{2}}\right)_{D_2} = 0$ for all $m$.

\end{proof}

\textbf{Case 2: When $N$ is even}

\begin{claim}
Take $x=\dfrac{3N}{2}+4$, $y=\dfrac{N-2}{2}$, and any $a\geqslant \dfrac{N+(N-1)(\frac{N}{2}+1)}{\frac{N}{2}-1}$.
\end{claim}

\begin{proof}
We need to show that 
$$\displaystyle I\left((Nam)^{\times \frac{3N}{2}+4},((N+2)am)^{\times \frac{N-2}{2}}\right)_{[(N+3)a+1]m-1} = 0, \forall m.$$

Applying Lemma \ref{lemma: reduction of multiplicity of points} with
\begin{align*}
    k&= (N-1)[(N+3)a+1]m-(N-1)- \left(\dfrac{N-2}{2}(N+2)am+\dfrac{N+4}{2}Nam\right) \\
    &= \left((N-1)-a\right) m-(N-1),
\end{align*}
(From here, we need $a\geqslant N$.)

for all $m$, we reduce to
$$I\left(A_1^{\times N+2},B_1^{\times \frac{N}{2}+2},C_1^{\times \frac{N-2}{2}}\right)_{D_1},$$

where $A_1= Nam, B_1= [(N-1)a+N-1]m-(N-1)$, $C_1= [(N+1)a+(N-1)]m-(N-1)$, and $D_1= [(N+2)a+N]m-N$.\par
\vspace{0.5em}

Applying Lemma \ref{lemma: reduction of multiplicity of points} again with
\begin{align*}
    k&= (N-1)D_1- \dfrac{N-2}{2}C_1-\dfrac{N+4}{2}A_1 \\
    &= \left((N-1)(\dfrac{N}{2}+1)-(\dfrac{N}{2}+1)a\right) m- (N-1)(\dfrac{N}{2}+1),
\end{align*}

for all $m$, we reduce to

$$I\left( A_2^{\times \frac{N}{2}},B_2^{\times \frac{N}{2}+2}, C_2^{\times \frac{N-2}{2}}, E_2^{\times \frac{N+4}{2}}\right)_{D_2},$$

where $A_2=Nam$ and $D_2=\left( (\dfrac{N}{2}+1)a+N+(N-1)(\dfrac{N}{2}+1) \right) m-N-(N-1)(\dfrac{N}{2}+1)$.\par
\vspace{0.5em}
Therefore, if we choose $a\geqslant \dfrac{N+(N-1)(\frac{N}{2}+1)}{\frac{N}{2}-1}$, we have 
$$Na \geqslant (\dfrac{N}{2}+1)a+N+(N-1)(\dfrac{N}{2}+1), \text{  hence,  } A_2>D_2,$$
thus, $I\left( A_2^{\times \frac{N}{2}},B_2^{\times \frac{N}{2}+2}, C_2^{\times \frac{N-2}{2}}, E_2^{\times \frac{N+4}{2}}\right)_{D_2}=0$ for all $m$. 
\end{proof}
This finishes the proof for the lemma.	
\end{proof}

Before proving the main result of this section, that is, the inequality $\ahat(I) > \frac{\reg(I)+N-1}{N}$ for generic points, we prove Lemma \ref{lemma: inductionlemma}, which is important in the inductive proof of the main Theorem \ref{theorem: lowerboundsforallnumber}. This lemma is a direct application of the results in \cite{localeffectivity}, where the authors studied Waldschmidt decomposition to investigate questions related to local effectivity, Waldschmidt constant, and Demailly's conjecture. We state their result here for our purpose.

\begin{lemma}[Theorem 4.1 \cite{localeffectivity}]\label{lemma: Waldschmidtdecomp}
Denote $\ahat(\PP^N, r)$ the Waldschmidt constant of the ideal of $r$ very general points in $\PP^N$. Let $N\geqslant 2$ and $k\geqslant 1$. Assume that for some integers $r_1,\ldots ,r_{k+1}$ and rational numbers $a_1,\ldots ,a_{k+1}$ we have 
$$\ahat(\PP^{N-1}, r_j) \geqslant a_j, \text{ for all } j=1,\ldots ,k+1,$$
$$k\leqslant a_j \leqslant k+1, \text{ for } j=1,\ldots ,k, \ \  a_1>k, \ \  a_{k+1}\leqslant k+1.$$
Then, $$\ahat(\PP^{N}, r_1+\ldots +r_{k+1}) \geqslant \left(1 - \sum_{j=1}^k \dfrac{1}{a_j} \right)a_{k+1} +k.$$
\end{lemma}

\begin{lemma}\label{lemma: inductionlemma}
	Assume that $N\geqslant 4$. Suppose that for each $2\leqslant \ell\leqslant  N-1$, the following is true:
	 $$  \text{ if }{N+\ell \choose N}<s \leqslant {N+\ell+1 \choose N} , \text{then for }  s \text{-many very general sets of points } \ahat(\PP^N, s) \geqslant \dfrac{N+\ell+1}{N}.$$  If for each of those $\ell$,  we set  $ s_\ell={N+1+\ell \choose N+1}+1 $, then for $s_\ell$ many generic points we have:
	 $$ \ahat(1^{ \times s_\ell})>\dfrac{N+1+\ell+1}{N+1}.$$
\end{lemma}
\begin{proof}
\noindent Let us write
\begin{align*}
s_\ell={N+1+\ell \choose N+1}+1 &= {N+\ell \choose N}+1+{N+\ell \choose N+1}\\
                                                               &={N+\ell \choose N}+1+ \dfrac{N+\ell}{N+1}{N+\ell-1 \choose N}\\
                                                               &=r_1+r_2
\end{align*}
where $r_1={N+\ell \choose N}+1 $ and $r_2=\frac{N+\ell}{N+1}{N+\ell-1 \choose N}$. 
For very general sets of points and $\ell \geqslant 2$, \\

$\ahat(\PP^N, r_1) \geqslant a_1=\dfrac{N+\ell+1}{N} =1+\dfrac{\ell+1}{N} $,  \\

$ \ahat(\PP^N, r_2) \geqslant   \ahat(\PP^N, {N+\ell-1 \choose N}+1)  \geqslant a_2= \dfrac{N+\ell-1+1}{N} =1+\dfrac{\ell}{N}$. \\

Since, $\ell \leqslant N-1$, then $1< a_1 \leqslant 2$ and $a_2\leqslant 2$. 
This satisfies the hypothesis of Lemma \ref{lemma: Waldschmidtdecomp} with  $k=1$. Thus for $s_l$ many very general points  in $\PP^{N+1}$, we have 
$$\ahat(\PP^N, r_1+r_2) \geqslant (1-\frac{1}{a_1})a_2+1 =\frac{(\ell+1)(N+\ell)}{(N+\ell+1)N}+1 > 1+\frac{\ell+1}{N+1},$$
which gives that for very general points $$\ahat(\PP^{N+1}, s_\ell) > 1+\frac{\ell+1}{N+1}.$$
Now we know that the Waldschmidt constant of the ideal  generic points is greater than that of very general points. Thus for $s_\ell$ many generic points in $\PP^{N+1}$, we have 
$$\ahat(1^{ \times s_\ell}) > 1+\frac{\ell+1}{N+1} =\frac{N+1+\ell+1}{N+1}.$$
                                          
\end{proof}
We are now ready to prove the main result of this section.
\begin{theorem}\label{theorem: lowerboundsforallnumber}
Let $I$ be the defining ideal of any number of $s$ generic points in $\PP^N$ where $s \geqslant N+4$. Then $$\ahat(I) > \frac{\reg(I)+N-1}{N}.$$    
\end{theorem}

\begin{proof}
We use induction on $N$. The base case $N=4$ follows from results in \cite{bghn2021chudnovskys} and Lemma \ref{lemma: lowers bounds on Waldschmidt in P4 for 8<s< 81 } as mentioned above. Suppose that the inequality holds for $N\geqslant 4$, we want to show the inequality hold for $N+1$. First, note that by combining the results in \cite{bghn2021chudnovskys}, Lemmas \ref{lemma: lowers bounds on Waldschmidt in P5678 for 2^N<s< 3^N}, \ref{lemma: combinatorialineq}, \ref{lemma: lowers bounds on Waldschmidt in PN for N+4<s< N+2chooseN}, and  \ref{lemma: lowers bounds on Waldschmidt in PN for N+2chooseN<s< N+3chooseN} applying to $\PP^{N+1}$, it suffices to prove the inequality
$$\ahat(I) > \frac{\reg(I)+N}{N+1},$$
for $s$ many generic points in $\PP^{N+1}$ where ${N+1+3 \choose N+1}+1 \leqslant s \leqslant {N+1+N+1-1 \choose N+1}$ and $N+1 \geqslant 5$. Moreover, it is enough to show that 
$$ \ahat(1^{ \times s_\ell})>\dfrac{N+1+\ell+1}{N+1},$$
for $ s_\ell={N+1+\ell \choose N+1}+1 $ generic points in $\PP^{N+1}$ where $3\leqslant \ell \leqslant N$. By the inductive hypothesis that the inequality $\ahat(I) > \frac{\reg(I)+N-1}{N}$ is true for any $s$ generic points in $\PP^N$, specializing the points, we have the inequality $\ahat(\PP^N, s) \geqslant \dfrac{N+\ell+1}{N}$ for any $s$ very general sets of points in $\PP^N$ where ${N+\ell \choose N}<s \leqslant {N+\ell+1 \choose N}$ and $2\leqslant \ell\leqslant  N-1$. By Lemma \ref{lemma: inductionlemma}, we finish the induction step.
\end{proof}

\begin{remark}
    As mentioned before, we only work with $N\geqslant 4$ and $s\geqslant N+4$ many points. The above inequality is stronger than Chudnovsky's inequality, hence, by specializing the points, we yield Chudnovsky's conjecture for $s\geqslant N+4$ very general points when $N\geqslant 4$, thus recover the main result in \cite{FMX2018} in this case.
\end{remark}

\section{Stable containment and Chudnovsky's Conjecture for general points}\label{section: Chudnovsky}
In this section, we extend our results in \cite{bghn2021chudnovskys} on stable Harbourne-Huneke Containment and Chudnovsky's conjecture for any numbers of general points in $\PP^N$.

\begin{lemma}\label{lemma: stronger HH for generic points}
	Let $I=I(\z)$ be the defining ideal of a generic sets of any number of points in $\PP^N_{\kk(\z)}$. Then $I$ satisfies the following containment $$ I^{(Nr-N)} \subset \mm_\z^{Nr}I^r, \text{ for } r \gg 0 .$$
\end{lemma}
\begin{proof}
	The proof follows the same pattern as \cite[Lemma 4.7]{bghn2021chudnovskys}, we present it in details for readers' convenience. By Theorem \ref{theorem: lowerboundsforallnumber} we have, $\ahat(I) > \dfrac{\reg(I)+N-1 }{N} $, hence, 
	$N\ahat(I) \geqslant \dfrac{r}{r-1}(\reg(I)+N-1)$, for $r \gg 0$, therefore,  
	$$\alpha(I^{Nr-N}) \geqslant (Nr-N)\ahat(I) \geqslant r(\reg(I)+N-1).$$ 
	Hence by \cite[Lemma 2.3.4]{BoH}, we get, $I^{(Nr-N)} \subset \mm_\z^{Nr}I^r, \text{ for } r \gg 0 $. 
\end{proof}

The following is the main Theorem of this section. 
\begin{theorem}\label{theorem: Containment for general points}
	Let $I$ be the defining ideal of a general set of any $s$ number of points in $\PP^N$. Then, there is a constant $r(s,N)$, depending only on $s$ and $N$, such that the stable containment $I^{(Nr)} \subseteq \mm^{(N-1)r}I^r$ holds when $r \geqslant r(s,N)$. 
\end{theorem}

The proof follows the same pattern as \cite[Theorem 4.8]{bghn2021chudnovskys}.  
\begin{proof}
	Let $I(\z)$ be the defining ideal of the set of $s $ generic points in $\PP^N_{\kk(\z)}$. By Lemma \ref{lemma: stronger HH for generic points}, there exists a constant $c \in \NN$ such that
	$$I(\z)^{(Nc-N)} \subseteq \mm_\z^{c(N-1)}I(\z)^c.$$
	By \cite[Satz 2 and 3]{Krull1948}, there exists an open dense subset $U \subseteq \AA^{s(N+1)}$ such that $\forall \a \in U$,
	$$\pi_\a(I(\z))^{(Nc-N)} = I(\a)^{(Nc-N)}, \quad \pi_\a(I(\z)^c) = I(\a)^c \quad \text{ and } \quad \pi_\a(\mm_\z^{c(N-1)}) = \mm^{c(N-1)}.$$
	Thus, for all $\a \in U$, we have
	\begin{align}
		I(\a)^{(Nc-N)} \subseteq \mm^{c(N-1)}I(\a)^c. \label{eq.HaHu110}
	\end{align}
	Note that we can pick $U$ such that for all $\a \in U$, we also have $\alpha(I(\a)) = \alpha(I(\z))$.
	
	Applying \cite[Corollary 3.2, and Remark 3.3]{bghn2021chudnovskys}, we get that, for $\a \in U$ and $r \gg 0$ (independent of $I(\a)$),
	$$I(\a)^{(Nr-N)} \subseteq \mm^{r(N-1)}I(\a)^r.$$ 
\end{proof}
\begin{remark}
	If $I$ is an ideal defining a set of general points, then they do satisfy the following containment 
	$$I^{(Nr-N+1)} \subseteq \mm^{(r-1)(N-1)}I^r, r \gg0 .$$	
	By combining with the results for sufficiently large numbers of points in \cite{bghn2021chudnovskys}, we have completed the proof showing the stable Harbourne-Huneke containment $I^{(Nr)} \subseteq \mm^{(N-1)r}I^r$, the above stronger containment, or in particular, the stable Harbourne containment $I^{(Nr-N+1)} \subseteq I^r, r \gg 0 $ for any set of general points in $\PP^N$. 
\end{remark}

\begin{theorem}\label{theorem: Chudnovsky for small points}
	If $I$ is the defining ideal of a general set of points in $\PP^N_\kk$, then $I$ satisfies Chudnovsky's conjecture, i.e., 
	$$\ahat(I) \geqslant \dfrac{\alpha(I)+N-1}{N}. $$
\end{theorem}

\begin{proof}
	By Theorem \ref{theorem: Containment for general points} we get 
	$$I^{(Nr)} \subseteq \mm^{Nr}I^r, \text{ for } r \gg 0.$$
	By taking the initial degree in each side, $$\alpha \big(I^{(Nr)} \big) \geqslant r(N-1)+r\alpha(I) .$$
	After dividing by $Nr$, and taking limit as $r \to \infty$, we get 
	$\ahat \big( I\big) \geqslant \dfrac{\alpha(I)+N-1}{N}.$
\end{proof}

\begin{remark}
	Incorporating with the results in \cite{bghn2021chudnovskys}, we are able to show the stable Harbourne-Huneke containment and Chudnovsky's conjecture for any number of general points in $\PP^N$ for all $N\geqslant 4$. Combining with the results in \cite{HaHu} for $N=2$ and  and in \cite{dumnicki2012symbolic, Dumnicki2015} for $N=3$, the stable Harbourne-Huneke containment and Chudnovsky's conjecture for any number of general points in $\PP^N$ for all $N$. It is still wide open whether each of the conjectures holds for any set of points. There are only a few affirmed answers, one is that Chudnovsky's conjecture holds for $s \leqslant {N+2 \choose N}-1$ number of points (see \cite{FMX2018}), and one is for all number of points in $\PP^2$ \cite{HaHu}. Abu Thomas informed us that he also used Cremona transformation technique to prove Chudnovsky's conjecture for $15$ linearly general points in $\PP^4$ see \cite[Theorem 5.1.5]{Abu}. The preliminaries work using the method in this paper, which proved the conjectures for any number of general points $N\leqslant 9$, can be found in the second author's thesis \cite{Thaithesis}.
\end{remark}

\bibliographystyle{alpha}
\bibliography{References}

\end{document}